\documentclass[letterpaper,11pt]{amsart}

\usepackage[T1]{fontenc}
\usepackage[utf8]{inputenc}
\usepackage[english]{babel}
\usepackage{a4wide}
\usepackage{amsfonts}
\usepackage{amssymb,,amsmath}
\usepackage{graphicx}
\usepackage{stmaryrd}
\usepackage{bbm} 
\usepackage{amsthm}
\usepackage{float} 
\usepackage{hyperref}
\usepackage{comment}
\hypersetup{
    colorlinks=true,       
    linkcolor=red,         
    citecolor=blue,        
    filecolor=magenta,     
    urlcolor=cyan          
}

\newtheoremstyle{note}{} {}{\itshape}{-6pt}{\bf}{. --}{ }{}

\theoremstyle{note}
\newtheorem{theorem}{Theorem}
\newtheorem{proposition}{Proposition}[section]

\newtheorem{lemma}[proposition]{Lemma}

\newtheorem{assumption}{Assumption}
\numberwithin{equation}{section}

\newenvironment{rem}{{\noindent\bf Remark }}{}

\usepackage[dvipsnames]{xcolor}
\usepackage{tikz}

\def\begi{\begin{itemize}}
	\def\endi{\end{itemize}}

\DeclareMathOperator{\dist}{dist}

\DeclareMathOperator{\vol}{vol}

\DeclareMathOperator{\Div}{div}

\title[Spectral estimate on the hyperbolic half-plane]{Spectral estimate for the Laplace--Beltrami operator on the hyperbolic half-plane}
\author{Marc Rouveyrol}
\address{Université Paris-Saclay, CNRS, Laboratoire de mathématiques d’Orsay, 91405, Orsay, France}
\email{marc.rouveyrol@universite-paris-saclay.fr}

\begin{document}

\begin{abstract}
The purpose of this note is to investigate the concentration properties of spectral projectors on manifolds.
This question has been intensively studied (by Logvinenko--Sereda, Nazarov, Jerison--Lebeau, Kovrizhkin, Egidi--Seelmann--Veseli{\'c}, Burq--Moyano, among others) in connection with the uncertainty principle. 
We provide the first high-frequency results in a geometric setting which is neither Euclidean nor a perturbation of Euclidean.
Namely, we prove the natural (and optimal) uncertainty principle for the spectral projector on the hyperbolic half-plane.
\end{abstract}


\maketitle
\tableofcontents

\newpage

\section{Introduction}

\subsection{Assumption and main result}

In the present note, we prove spectral estimates from thick sets on the hyperbolic half-plane $\mathbb{H}^2 = \{(x,y) \in \mathbb{R}^2, y > 0\}$, endowed with its usual metric $g = \frac{1}{y^2}(dx^2 + dy^2)$. The Laplace--Beltrami operator on this manifold is $\Delta_g = y^2 \left(\partial_x^2 + \partial_y^2\right)$, and the Riemannian measure corresponding to the metric $g$ is $\frac{dxdy}{y^2}$. Recall that the operator $-\Delta_g$ on $L^2(\mathbb{H}^2, \frac{dxdy}{y^2})$ with domain $H^2_g(\mathbb{R}^d)$ is self-adjoint with spectrum $[\frac{1}{4},+\infty)$ (we refer to \cite{hebey1996sobolev,grigor'yan2009book} for definitions and properties of Sobolev spaces on Riemannian manifolds and to \cite{bergeron2011livre}, \cite[Chapter 4]{borthwick2016book} for spectral theory of the hyperbolic Laplacian).
Thus, by the spectral theorem, it is possible to write $$u = \int_\frac{1}{4}^{+\infty} dm_\lambda u, \quad \Pi_\Lambda u = \mathbbm{1}_{\sqrt{-\Delta_g} \leq \Lambda} u =\int_\frac{1}{4}^\Lambda dm_\lambda u,$$ where $dm_\lambda$ denotes the spectral measure of the operator $\sqrt{-\Delta_g}$ defined by functional calculus. For $\phi, \varphi$ two bounded, real-valued functions over $\mathbb{R}$, one has \begin{multline*}\phi(\sqrt{-\Delta_g})u = \int_\frac{1}{4}^{+\infty}\phi(\lambda)dm_\lambda u, \\ \langle \phi(\sqrt{-\Delta_g})u, \varphi(\sqrt{-\Delta_g})u \rangle_{L^2\left(\mathbb{H}^2;\frac{dxdy}{y^2}\right)} = \int_\frac{1}{4}^{+\infty} \phi(\lambda) \varphi(\lambda) (dm_\lambda u, u)\end{multline*} so that \begin{equation}
    \label{eq: functional calculus bound}
    \|\phi(\sqrt{-\Delta_g})\Pi_\Lambda u\|_{L^2\left(\mathbb{H}^2; \frac{dxdy}{y^2}\right)} \leq \sup_{\lambda \in [\frac{1}{4},\Lambda]} |\phi(\lambda)| \|\Pi_\Lambda u\|_{L^2\left(\mathbb{H}^2;\frac{dxdy}{y^2}\right)}.
\end{equation}

Geodesic balls for the metric $g$ are given by \begin{equation}
    \label{eq: geodesic balls}
    \forall (x,y) \in \mathbb{H}^2, r > 0, \quad B^g_r(x,y) = B_{y\sinh(r)}(x,y\cosh(r)) 
\end{equation} where $B_r^g$ denotes a geodesic ball of center $(x,y)$ and radius $r$ and $B_{y\sinh(r)}$ a Euclidean ball. Throughout the paper, we will consider a subset $\omega$ of $\mathbb{H}^2$ which satisfies the following assumption:

\begin{assumption}
    \label{ass: geodesic balls}
    There exist $\delta, R > 0$ such that for any $(x,y) \in \mathbb{H}^2$, there holds $$\vol_g(\omega \cap B^g_R(x,y)) = \int_{\omega \cap B^g_R(x,y)} \frac{dxdy}{y^2} \geq \delta.$$
\end{assumption}

It essentially states that the measure of the set $\omega$ is large enough and sufficiently spread out so as to be uniformly bounded from below in every geodesic ball, \emph{i.e.} that $\omega$ is a \emph{thick set} for the Riemannian metric and density (see \cite{logvinenkosereda1974equivalent, kovrizhkin2001relatedtoLog-Ser}). Note that in the literature, this lower usually concerns the ratio $\frac{\vol_g(\omega \cap B^g_R(x,y))}{\vol_g(B^g_R(x,y))}$. Since all balls of radius $R$ have equal volume in our setting, we omit this divider by including it in $\delta$.

Our main result is that this assumption is equivalent to a spectral estimate from set $\omega$, as follows.

\begin{theorem}
    \label{theo: main result}
    Consider some Borel set $\omega \subset \mathbb{H}^2$ and define the spectral projector $\Pi_\Lambda$ as above, then the following conditions are equivalent: \begin{enumerate}
        \item $\omega$ satisfies Assumption \ref{ass: geodesic balls} for some $\delta, R > 0$;
        \item there exists some constant $C(\delta, R) > 0$ such that for any $\Lambda \geq \frac{1}{4}, u \in L^2\left(\mathbb{H}^2, \frac{dxdy}{y^2}\right)$, \begin{equation}
        \label{eq: main result estimate}
        \|\Pi_\Lambda u\|_{L^2\left(\mathbb{H}^2;\frac{dxdy}{y^2}\right)} \leq C e^{C\Lambda} \|\Pi_\Lambda u\|_{L^2\left(\omega; \frac{dxdy}{y^2}\right)}.
    \end{equation}
    \end{enumerate}
\end{theorem}

\subsection{Some bibliography on spectral estimates}

The first equivalence between a spectral estimate (formulated as an uncertainty principle for the Fourier transform) and a thickness condition on the observability set was obtained by Logvinenko and Sereda using complex analysis methods \cite{logvinenkosereda1974equivalent} (see also Zygmund \cite[pp.202-208] {zygmund2002trigonometricseries}, Nazarov \cite{nazarov1994local} and Kovrizhkin \cite{kovrizhkin2001relatedtoLog-Ser, kovrizhkin2003uncertainty}). The idea of using propagation of smallness inequalities for harmonic functions to prove these estimates can be traced back to Jerison and Lebeau \cite{jerisonlebeau1999nodal}, with applications to bounds on the Hausdorff measure of nodal sets.

These spectral estimates have had fruitful applications to controllability (or equivalently observability) of the heat equation and of parabolic systems since the works of Lebeau, Robbiano and Zuazua \cite{lebeaurobbiano1995controlechaleur, lebeauzuazua1998thermoelasticity}. In \cite{miller2010directlebeaurobbiano}, Miller provided an abstract framework for proving observability inequalities at optimal cost for parabolic systems. Phung and G. Wang generalized these observability estimates to any positive-measure time set using a telescoping series method \cite{phungwang2013observability}, and Apraiz, Escauriaza, G. Wang and C. Zhang \cite{apraizescauriaza2013nullcontrol, apraizescauriazawangzhang2014observability} proved observability from any positive measure time-space set for the heat equation on bounded, star-shaped domains.

On unbounded domains, Egidi and Veseli{\'c} \cite{egidiveselic2018controlheat} and C. Wang, M. Wang, C. Zhang and Y. Zhang \cite{wangwangzhangzhang2019spectralineq_heat} simultaneously proved the equivalence between observability of the (constant-coefficient) heat equation and thickness of the observation set on the whole space. The former used Kovrizhkin's formulation of the Logvinenko--Sereda theorem \cite{kovrizhkin2000fourier,kovrizhkin2001relatedtoLog-Ser}, while the latter's proof is based on the equivalence with Jerison--Lebeau type estimates for spectral projectors and the telescopic series method of \cite{phungwang2013observability}. The result of \cite{wangwangzhangzhang2019spectralineq_heat} can be seen as the flat-Laplacian equivalent of Theorem \ref{theo: main result}. We refer to the works of Beauchard, Jaming, Martin and Pravda-Starov \cite{beauchardjamingpravda-starov2021spectral, martinpravda-starov2023spectral}, Egidi, Seelmann and their collaborators \cite{egidiseelmann2021Logv-Seredatype, egidinakicseelmann+2020heatequation} and the references therein for further spectral estimate and parabolic controllability results from thick sets.

Studies about the interaction between spectral estimates for the Laplacian and the ambient geometry are more scarce in the literature. In \cite{lebeaumoyano2019spectralinequalities}, Lebeau and Moyano proved spectral estimates for Schrödinger operators with analytic potentials, when the metric is an analytic perturbation of the identity. More recently, Burq and Moyano proved spectral estimates for Laplace operators on compact, $W^{2,\infty}$ manifolds with Lipschitz metric and density \cite{burqmoyano2023propagationheat}, combining ideas of Jerison and Lebeau \cite{jerisonlebeau1999nodal} with propagation of smallness results of Logunov and Malinnikova \cite{logunovmalinnikova2018qtttve_propag_smallness}. They then showed that the thickness condition for subsets of $\mathbb{R}^d$ implies spectral projector estimates for uniformly elliptic Laplacians with Lipschitz metric and density \cite{burqmoyano2021spectralestimates}.

The aforementioned propagation of smallness result \cite[Theorem 5.1]{logunovmalinnikova2018qtttve_propag_smallness} is a three-sphere type estimate for the gradients of solutions of homogeneous elliptic equations in divergence form. It plays a central role in the present article. A remarkable feat of this result is that the smaller set in the estimate is of Hausdorff dimension smaller than $d-1$, where $d$ denotes the dimension of the full space. We refer to \cite{logunov2018nodalpolynomial,logunov2018nodalnadirashvili,logunovmalinnikova2018qtttve_propag_smallness} for details concerning the spirit and proofs of this result. For the convenience of the reader, a statement of the estimate is given in Appendix \ref{sec: Appendix Logunov}.

To the best of our knowledge, the first spectral estimate results on non-compact manifolds with a geometry that is not a perturbation of the Euclidean one are those of Rose and Tautenhahn \cite{rosetautenhahn2023unique} on manifolds with Ricci curvature bounded from below. The main difference with Theorem \ref{theo: main result} is that their estimate holds for values of $\Lambda$ below a certain threshold given by the geometry.

The main contribution of the present work is to give the first (again to the best of our knowledge) example of a high-frequency spectral estimate from thick sets on a non-compact manifold with a geometry that is not a perturbation of the Euclidean one. The proof is inspired by the works of Burq--Moyano \cite{burqmoyano2023propagationheat, burqmoyano2021spectralestimates}, in which the link between the Logunov--Malinnikova inequality and spectral estimates was established, and the article \cite{wangwangzhangzhang2019spectralineq_heat}, where the authors use an observability inequality to recover a thickness condition on the observation set.

\subsection{Scheme of the proof}

The rest of the paper is dedicated to proving Theorem \ref{theo: main result}. In the proof that Assumption \ref{ass: geodesic balls} implies the spectral estimate, it is more natural to consider Euclidean rectangles than geodesic balls, hence the following remark: 

\medskip

\begin{rem}
Consider $R'$ given by the one-to-one identity $$\tanh(R) = \min\left(1-2^{-R'}, \left(\frac{3}{2}\right)^{R'} - 1\right),$$ where $R$ is given by Assumption \ref{ass: geodesic balls}. Then, every rectangle $$\mathcal{R}_{j,k}(R'):= (2^{R'j} k - 2^{R'j}, 2^{R'j}k + 2^{R'j})_x \times (2^{R'(j-1)}, 3^{R'} \times 2^{R'(j-1)})_y,$$ contains some geodesic ball of diameter $R$. Precisely, \eqref{eq: geodesic balls} gives \begin{equation} \label{eq: rectangle contains hyp ball}
    B^g_R\left(2^{R'j}k, \frac{2^{R'j}}{\cosh(R)}\right) \subset \mathcal{R}_{j,k}(R').
\end{equation}
\end{rem} By \eqref{eq: rectangle contains hyp ball}, Assumption \ref{ass: geodesic balls} implies the following assumption, which is the one used in the proof:

\begin{assumption}
    \label{ass: Cjk}
    There exists some $\delta > 0$ such that for any $(j,k) \in \mathbb{Z}^2$, there holds $$\int_{\omega \cap \mathcal{R}_{j,k}(R')} \frac{dxdy}{y^2} \geq \delta.$$
\end{assumption}

It is easily checked that \begin{equation}\label{eq: covering} \mathbb{H}^2 \subset \bigcup_{(j,k) \in \mathbb{Z}^2} \mathcal{R}_{j,k}\end{equation} and that for any $x \in \mathbb{H}^2$, there holds $$\sharp \{(j,k), x \in \mathcal{R}_{j,k}\} \leq N$$ for some $N(R')$ independent of $x$. Thus, \begin{equation}
    \label{eq: equivalence of norms}
    \|.\|_{L^2(\mathbb{H}^2; \frac{dxdy}{y^2})}^2 \leq \sum_{j,k} \|.\|_{L^2(\mathcal{R}_{j,k}; \frac{dxdy}{y^2})}^2 \leq N \|.\|_{L^2(\mathbb{H}^2; \frac{dxdy}{y^2})}^2.
\end{equation}

The strategy of proof of the sufficient condition (namely that Assumption \ref{ass: geodesic balls} implies the spectral estimate) adapts Burq and Moyano's \cite{burqmoyano2023propagationheat} on the whole space: we apply Logunov and Malinnikova's propagation of smallness inequality to a well-chosen harmonic function $v_\Lambda(x,y,t)$, with $t$ an auxiliary variable and $(x,y)$ localized inside each $\mathcal{R}_{j,k}$. We then use a Sobolev injection-type lemma to recover the spectral projector inequality \eqref{eq: main result estimate}. The main novelty arises from the fact that the operator $-\Delta_g = y^2(\partial_x^2 + \partial_y^2)$ is not uniformly elliptic over $\mathbb{H}^2$, hence the covering $(\mathcal{R}_{j,k})$ of the half-plane is chosen so that the metric $g$ does not vary too much over each subset. Besides, the multiplicative constant in Logunov and Malinnikova's estimate depends on the ellipticity and Lipschitz constants of the operator considered. To avoid dealing with this dependence, we perform a change of coordinates over every subset $\mathcal{R}_{j,k}$ to obtain constants uniform with respect to $(j,k)$.

The proof of the necessary condition is inspired by similar work by G. Wang, M. Wang, C. Zhang and Y. Zhang for the flat Laplacian on $\mathbb{R}^d$. In \cite{wangwangzhangzhang2019spectralineq_heat}, they deduce thickness of the set $\omega$ from a final-time observability estimate that has been known to derive from spectral projector estimates since the works of Miller \cite{miller2010directlebeaurobbiano} and Phung and G. Wang \cite{phungwang2013observability}. We adapt this strategy to our framework using upper and lower Gaussian (in terms of the geodesic distance) bounds for the hyperbolic heat kernel.

The next section is split into three parts. In the first two, we prove that Assumption \ref{ass: Cjk} implies the spectral inequality \eqref{eq: main result estimate}. In Subsection \ref{sec: proof main result}, we apply the Logunov--Malinnikova and Sobolev-type estimates to prove \eqref{eq: main result estimate}. Subsection \ref{sec: proof Sobolev injection} is dedicated to proving the Sobolev injection lemma. In subsection \ref{sec: necessary condition}, we show that the spectral inequality \eqref{eq: main result estimate} (or rather its corollary, observability inequality \eqref{eq: observability}) implies Assumption \ref{ass: geodesic balls}. The appendices do not contain any new mathematical content and are given for the sake of self-containedness. In Appendix \ref{sec: Appendix Logunov}, we state the Logunov--Malinnikova inequality which is used in subsection \ref{sec: proof main result}. In Appendix \ref{sec: Appendix observability}, we reproduce a proof that the spectral estimate \eqref{eq: main result estimate} implies observability inequality \eqref{eq: observability}, following \cite{wangwangzhangzhang2019spectralineq_heat}.

We conclude this introduction with a few notations. Throughout the rest of the article, we denote $z=(x,y)$ a generic element in the hyperbolic half-plane $\mathbb{H}^2$, $d_g(z,z')$ the geodesic distance between two points $z, z' \in \mathbb{H}^2$ and $d\vol_g:= d\vol_g(z) = \frac{dxdy}{y^2}$ the Riemannian measure corresponding to the metric $g$. We denote by $\|.\|_{L^2_g}$ norms involving the measure $d\vol_g(z)$ (or the measure $dtd\vol_g(z)$ when also integrating over the auxiliary variable $t$). Norms denoted $\|.\|_{L^2}$ without the index $g$ are integrated using Lebesgue measure. The symbol $\lesssim$ means that the left-hand-side is bounded by the right-hand-side up to an implicit multiplicative constant which is independent of relevant quantities. The dependence of implicit constants is specified when deemed necessary. $A \simeq B$ means that both $A \lesssim B$ and $B \lesssim A$ hold.

\section{Proof of Theorem \ref{theo: main result}}

\subsection{Proof of the spectral estimate}
\label{sec: proof main result}

We prove inequality \eqref{eq: main result estimate} under Assumption~\ref{ass: Cjk}. For the sake of simplicity, the proof is written for $R' = 1$. It can be adapted easily to the case $R' > 0$ by adding appropriate $R'$ exponents in the bounds of the rectangles considered.

Take $u \in L^2_g(\mathbb{H}^2)$ and define $$v_\Lambda(t, x, y) = \int_{\lambda = 0}^\Lambda \frac{\sinh(\lambda t)}{ \lambda} dm_\lambda u.$$ Then $-\Delta_g v_\Lambda = \int_0^\Lambda \frac{\sinh(\lambda t)}{\lambda} \lambda^2 dm_\lambda u$ so that $$(\partial_t^2 + \Delta_g) v_\Lambda = 0.$$

Recalling that $\mathcal{R}_{j,k} = \{(x,y) \in \mathbb{H}^2, |x - 2^j k| < 2^j, 2^{j-1} < y < 3 \times 2^{j-1} \}$ for any $(j, k) \in \mathbb{Z}^2$ and $\mathbb{H}^2 \subset \cup_{j,k} \mathcal{R}_{j,k}$, let $1 = \sum_{(j,k) \in \mathbb{Z}^2} \chi_{j,k}$ be a partition of unity associated to this covering.  Over each rectangle $\mathcal{R}_{j,k}$, we perform the change of variables $$(X, Y) = \phi_{j,k}(x,y) = (2^{-j}x - k, 2^{-j}y)$$ so that $(X, Y)$ spans the rectangle $\mathcal{R} = (-1, 1)_X \times (\frac{1}{2}, \frac{3}{2})_Y$ when $(x, y)$ spans $\mathcal{R}_{j,k}$. This change of coordinates preserves the hyperbolic Laplacian: $$\Delta_g = y^2(\partial_x^2 + \partial_y^2) = 2^{2j}Y^2(2^{-2j}\partial_X^2 + 2^{-2j}\partial_Y^2) = Y^2(\partial_X^2 + \partial_Y^2).$$ We then define $V_\Lambda^{j,k}(t, X, Y) = v_\Lambda(t, \phi_{j,k}^{-1}(X,Y))$, which satisfies $$(\partial_t^2 + Y^2 \Delta_{X,Y}) V_\Lambda^{j,k}(t, X, Y) = (\partial_t^2 + \Delta_g) v_\Lambda(t, x, y) = 0, \quad (X, Y) \in \mathcal{R}.$$

Multiplying by $\frac{1}{Y^2}$, we get $$(\partial_t \frac{1}{Y^2} \partial_t + \Delta_{X,Y})V_\Lambda^{j,k} = 0 \text{ over } \mathcal{R},$$ which in turn can be written $\Div(A \nabla_{t,X,Y} V_\Lambda^{j,k}) = 0$ for $$A(t,X,Y) = \begin{pmatrix}
    \frac{1}{Y^2} & 0 & 0 \\
    0 & 1 & 0 \\
    0 & 0 & 1
\end{pmatrix}.$$ It is clear that for any $T > 0$, $A$ is uniformly elliptic with Lipschitz coefficients over the box $(-T, T)_t~\times~\mathcal{R}~=~(-T, T)_t~\times~(-1, 1)_X~\times~(\frac{1}{2}, \frac{3}{2})_Y$, and its ellipticity and Lipschitz constants are independent of $T, j, k$.

Fix $0 < T_1 < T_2$ and $\mathcal{R}_2 = (-2, 2)_X \times (\frac{1}{4}, \frac{7}{4})_Y$. We denote \begin{multline}
    \label{eq: logunov-malinnikova sets (t, X, Y)} \mathcal{K} = (-T_1, T_1) \times \mathcal{R}, \Omega = (-T_2, T_2) \times \mathcal{R}_2, F_{j,k} = \phi_{j,k}(\omega) \cap \mathcal{R}, E_{j,k} = \{0\} \times F_{j,k}.
\end{multline} Then by Assumption \ref{ass: Cjk} $$\int_{\omega \cap \mathcal{R}_{j,k}} \frac{dxdy}{y^2} = \int_{\phi_{j,k}(\omega) \cap \mathcal{R}} \frac{dXdY}{Y^2} \geq \delta.$$ Thus, \begin{equation}
\label{eq: lower bound after scaling}
\vol_{euc}(\phi_{j,k}(\omega) \cap \mathcal{R}) \geq \frac{\delta}{4}\end{equation} uniformly with respect to $j,k$. By Logunov and Malinnikova's theorem \cite[Theorem 5.1]{logunovmalinnikova2018qtttve_propag_smallness}, the statement of which is recalled in Appendix \ref{sec: Appendix Logunov}, we get \begin{equation}
    \label{eq: logunov malinnikova L infty bound}
    \sup_{\mathcal{K}} |\nabla_{t,X,Y} V_\Lambda^{j,k}| \leq C \sup_{E_{j,k}} |\nabla_{t,X,Y} V_\Lambda^{j,k}|^\alpha \sup_{\Omega} |\nabla_{t,X,Y} V_\Lambda^{j,k}|^{1-\alpha}.
\end{equation}  The constants $C > 0, 0 < \alpha < 1$ depend on $\delta$ but do not depend on $u$, $\Lambda$, $j$ or $k$.

We change the $L^\infty$ norms to $L^2$ norms using a trick taken from \cite[Section 2]{burqmoyano2021spectralestimates}. Assume that $\nabla V_\Lambda^{j,k}|_\mathcal{K}$ is not identically zero and let $$a = \left( \epsilon \frac{\sup_{\mathcal{K}} |\nabla_{t,X,Y} V_\Lambda^{j,k}|}{\sup_{\Omega} |\nabla_{t,X,Y} V_\Lambda^{j,k}|^{1-\alpha}} \right)^\frac{1}{\alpha}, \quad F'_{j,k} = \{(X,Y) \in F_{j,k}, |\nabla_{t,X,Y} V_\Lambda^{j,k}(t=0, X, Y)| \leq a\}.$$ Let $E'_{j,k} = \{0\} \times F'_{j,k}$. One necessarily has $\vol_{euc}(F'_{j,k}) \leq \frac{\delta}{8}$. Indeed, otherwise, the Logunov--Malinnikova estimate \eqref{eq: logunov malinnikova L infty bound} holds with identical constants but replacing $E_{j,k}$ with $E'_{j,k}$. This gives \begin{equation}
    \sup_{\mathcal{K}} |\nabla_{t,X,Y} V_\Lambda^{j,k}| \leq C a^\alpha  \sup_\Omega |\nabla_{t,X,Y} V_\Lambda^{j,k}|^{1-\alpha} \leq C\epsilon \sup_{\mathcal{K}} |\nabla_{t,X,Y} V_\Lambda^{j,k}|.
\end{equation} For $C\epsilon < 1$, we obtain that $\nabla_{t,X,Y} V_\Lambda^{j,k}|_\mathcal{K}$ is identically zero, hence a contradiction.

As a consequence, \begin{multline} \label{eq: L2 bound from below} \int_{F_{j,k}} |\nabla_{t,X,Y} V_\Lambda^{j,k}(t=0, X, Y)|^2 dX dY \geq \int_{F_{j,k} \setminus F'_{j,k}} |\nabla V_\Lambda^{j,k}(t=0, X, Y)|^2 dXdY \\ \geq a^2 \frac{\delta}{8} \geq \frac{\delta}{8} \left( \epsilon \frac{\sup_{\mathcal{K}} |\nabla_{t,X,Y} V_\Lambda^{j,k}|}{\sup_{\Omega} |\nabla_{t,X,Y} V_\Lambda^{j,k}|^{1-\alpha}} \right)^\frac{2}{\alpha}.\end{multline}

Since $\nabla_{t,X,Y} V_\Lambda^{j,k}(t=0) = \begin{pmatrix}\Pi_\Lambda u \circ \phi_{j,k}^{-1} \\ 0 \\ \vdots \\ 0\end{pmatrix}$,
denoting $U_\Lambda^{j,k} = \Pi_\Lambda u \circ \phi_{j,k}^{-1}$, we get $$\begin{aligned}\int_\mathcal{R} |U_\Lambda^{j,k}(X,Y)|^2 dXdY &\leq \sup_{t \in (-T_1, T_1)} \int_{\mathcal{R}} |\nabla_{t,X,Y} V_\Lambda^{j,k}|^2 dXdY \\ &\lesssim \sup_{\mathcal{K}} |\nabla_{t,X,Y} V_\Lambda^{j,k}|^2 \\ &\lesssim \left( \int_{F_{j,k}} |\nabla_{t,X,Y} V_\Lambda^{j,k}(t=0, X, Y)|^2 dX dY\right)^\alpha \sup_{\Omega}|\nabla_{t,X,Y} V_\Lambda^{j,k}|^{2-2\alpha} \\ & \lesssim \left( \int_{F_{j,k}} |U_\Lambda^{j,k}(X, Y)|^2 dX dY\right)^\alpha \sup_{\Omega}|\nabla_{t,X,Y} V_\Lambda^{j,k}|^{2-2\alpha}. \end{aligned}$$ The implicit multiplicative constant in the last inequality depends only on $\delta, \epsilon, T_1, T_2$ and numerical constants, and does not depend on $\Lambda, u, j, k$. Note that in the case where $R'$ (see Assumption \ref{ass: Cjk} and the preliminary remark of this section) is fixed and different from 1, this constant also depends on $R'$ through the lower bound in \eqref{eq: lower bound after scaling}, the definitions of $\mathcal{R}$ and $\mathcal{R}_2$, and the constants $\Lambda_1, \Lambda_2, \rho$ of Theorem \ref{theo: logunov-malinnikova}.

Recall that since $y$ is roughly of size $2^j$ over $\mathcal{R}_{j,k}$, we have for any function $f \in L^2(\mathcal{R})$, $$\int_\mathcal{R} |f(X, Y)|^2 dXdY = \int_{\mathcal{R}_{j,k}} 2^{-2j} |f \circ \phi_{j,k}(x,y)|^2 dxdy \simeq \int_{\mathcal{R}_{j,k}} |f \circ \phi_{j,k}(x,y)|^2 \frac{dxdy}{y^2}.$$ Applying this to $f = U_\Lambda^{j,k}$ and $f = U_\Lambda^{j,k} \mathbbm{1}_{\phi_{j,k}(\omega)}$, we get \begin{equation} \label{eq: density appears naturally}\|U_\Lambda^{j,k}\|_{L^2(\mathcal{R})} \simeq \|\Pi_\Lambda u\|_{L^2_g(\mathcal{R}_{j,k})}, \quad \|U_\Lambda^{j,k}\|_{L^2(F_{j,k})} \simeq \| \Pi_\Lambda u\|_{L^2_g(\mathcal{R}_{j,k} \cap \omega)}.\end{equation} Hence, \begin{equation}
    \label{eq: logunov malinnikova L2} \|\Pi_\Lambda u\|_{L^2_g(\mathcal{R}_{j,k})}^2 \leq C \|\Pi_\Lambda u\|_{L^2(\mathcal{R}_{j,k} \cap \omega)}^{2\alpha} \sup_{\Omega}|\nabla_{t,X,Y} V_\Lambda^{j,k}|^{2-2\alpha}.
\end{equation}

We now state the following Sobolev injection lemma, the proof of which is postponed to the next subsection.

\begin{lemma}
    \label{lemma: sobolev injection}
    There exist constants $C, K > 0$ independent of $u$ and $\Lambda$ such that for any $u \in L^2_g(\mathbb{H}^2)$, $$\sum_{j,k} \sup_\Omega|\nabla_{t,X,Y} V_\Lambda^{j,k}|^2 \leq C e^{K\Lambda} \|\Pi_\Lambda u\|_{L^2_g(\mathbb{H}^2)}^2.$$
\end{lemma}

By Young's inequality, $$a^{2-2\alpha} b^{2\alpha} \leq C(a^2 + b^2).$$ Thus, by \eqref{eq: logunov malinnikova L2}, for any $D > 0$, \begin{equation}
    \|\Pi_\Lambda u\|_{L^2_g(\mathcal{R}_{j,k})}^2 \leq C e^{-D\Lambda} \sup_\Omega |\nabla_{t,X,Y} V_\Lambda^{j,k}|^2 + C e^{\frac{2-2\alpha}{2\alpha} D\Lambda} \|\Pi_\Lambda u\|_{L^2_g(\omega \cap \mathcal{R}_{j,k})}^2.
\end{equation} Summing over $(j,k)$ and applying Lemma \ref{lemma: sobolev injection} then gives \begin{multline}
    \|\Pi_\Lambda u\|_{L^2_g(\mathbb{H}^2)}^2 \leq C e^{-D\Lambda} \sum_{j,k} \sup_{\Omega} |\nabla_{t,X,Y}V_\Lambda^{j,k}|^2 + C e^{\frac{2-2\alpha}{2\alpha}D\Lambda} \|\Pi_\Lambda u\|_{L^2_g(\omega \cap \mathcal{R}_{j,k})}^2 \\ \leq C e^{-D \Lambda} e^{K\Lambda} \|\Pi_\Lambda u\|_{L^2_g(\mathbb{H}^2)}^2 + C e^{\frac{2-2\alpha}{2\alpha} D\Lambda} \|\Pi_\Lambda u\|_{L^2_g(\omega)}^2.\end{multline} Fixing $D > K$ gives the result for $\Lambda$ greater than some $\Lambda_0$ satisfying $Ce^{(K-D)\Lambda_0} < 1$. For $\Lambda < \Lambda_0$, the result is proved by applying Theorem \ref{theo: main result} to $\Pi_\Lambda u$.

\subsection{Proof of Lemma \ref{lemma: sobolev injection}}
\label{sec: proof Sobolev injection}

We now prove Lemma \ref{lemma: sobolev injection}. Recall that $\Omega$ denotes the set $(-T_2, T_2)_t~\times~(-2, 2)_X~\times~\left(\frac{1}{4}, \frac{7}{4}\right)_y$ and $V_\Lambda^{j,k} = v_\Lambda \circ \phi_{j,k}$. By the Sobolev injection in dimension 3, we have $$\sup_\Omega | \nabla_{t,X,Y} V_\Lambda^{j,k}|^2 \leq \|\chi \nabla_{t,X,Y} V_\Lambda^{j,k} \|_{H^2(\Omega^{(1)})}^2$$ with $\chi$ a smooth cutoff function supported in $$\Omega^{(1)} = (-T_3,T_3) \times (-2-\eta,2+\eta) \times \left(\frac{1}{4}-\eta,\frac{7}{4}+\eta\right)$$ and equal to 1 over $\Omega$. $\eta$ denotes a positive, small enough constant and $T_3 > T_2$. We can then apply elliptic regularity (see \cite[Section 6.3]{evans2022book}) to the operator $-(\partial_t^2 + Y^2 \Delta_{X,Y})$, which is uniformly elliptic over $\Omega^{(1)}$, to obtain \begin{multline}
\label{eq: elliptic regularity}
\|\chi\nabla_{t,X,Y} V_\Lambda^{j,k} \|_{H^2(\Omega^{(1)})}^2 \lesssim \|(1 - \partial_t^2 - Y^2 \Delta_{X,Y}) \chi\nabla_{t,X,Y}V_\Lambda^{j,k} \|_{L^2(\Omega^{(1)})}^2. \\
\lesssim \|(1 - \partial_t^2 - Y^2 \Delta_{X,Y}) \chi \partial_t V_\Lambda^{j,k} \|_{L^2(\Omega^{(1)})}^2 + \|(1 - \partial_t^2 - Y^2 \Delta_{X,Y}) \chi \nabla_{X,Y}V_\Lambda^{j,k} \|_{L^2(\Omega^{(1)})}^2.
\end{multline} The first term is bounded up to some multiplicative constant by \begin{multline}
\label{eq: time derivative term bound 1}
\|\partial_t V_{\Lambda}^{j,k}\|_{L^2(\Omega^{(1)})}^2 + \|\partial_t^2 V_{\Lambda}^{j,k}\|_{L^2(\Omega^{(1)})}^2 + \|\partial_t^3 V_{\Lambda}^{j,k}\|_{L^2(\Omega^{(1)})}^2 \\+ \|\nabla_{X,Y} \partial_t V_\Lambda^{j,k}\|_{L^2(\Omega^{(1)})}^2 + \|(Y^2 \Delta_{X,Y}) \partial_tV_\Lambda^{j,k}\|_{L^2(\Omega^{(1)})}^2\end{multline} and the second term by \begin{multline}
    \label{eq: gradient term bound 1}
    \|\nabla_{X,Y}V_\Lambda^{j,k} \|_{L^2(\Omega^{(1)})}^2 + \|\nabla_{X,Y} \partial_t V_\Lambda^{j,k} \|_{L^2(\Omega^{(1)})}^2 + \|\nabla_{X,Y} \partial_t^2 V_\Lambda^{j,k} \|_{L^2(\Omega^{(1)})}^2 + \|Y^2 \nabla_{X,Y}V_\Lambda^{j,k} \|_{L^2(\Omega^{(1)})}^2 \\+ \|Y^2\nabla_{X,Y}^2 V_\Lambda^{j,k} \|_{L^2(\Omega^{(1)})}^2 + \|Y \Delta_{X,Y}V_\Lambda^{j,k} \|_{L^2(\Omega^{(1)})}^2 + \|\nabla_{X,Y} (Y^2 \Delta_{X,Y} V_\Lambda^{j,k}) \|_{L^2(\Omega^{(1)})}^2
\end{multline} where for the last four terms we used the identity $$Y^2\Delta(\chi \nabla f) = Y^2 (\Delta \chi) \nabla f + Y^2 \nabla^2 f (\nabla \chi) - \chi \Delta f \begin{pmatrix} 0 \\ 2Y \end{pmatrix}+ \chi \nabla(Y^2 \Delta f)$$ and $\nabla^2 f$ denotes the Hessian matrix of a function $f$. We again swell $\Omega^{(1)}$ by $\eta$ in the $(X,Y)$ variables to bound the Hessian of $V_\Lambda^{j,k}$ by its Laplacian and lower order terms: $$\|\nabla_{X,Y}^2 V_\Lambda^{j,k}\|_{L^2(\Omega^{(1)})}^2 \lesssim \|V_\Lambda^{j,k}\|_{L^2(\Omega^{(2)})}^2 + \|\nabla_{X,Y} V_\Lambda^{j,k}\|_{L^2(\Omega^{(2)})}^2 + \|\Delta_{X,Y} V_\Lambda^{j,k}\|_{L^2(\Omega^{(2)})}^2$$ with $\Omega^{(2)} = (-T_3,T_3) \times (-2-2\eta,2+2\eta) \times \left(\frac{1}{4}-2\eta,\frac{7}{4}+2\eta\right)$ and an implicit multiplicative constant independent of $(j,k, \Lambda)$.

For $(j,k) \in \mathbb{Z}^2$, we denote $\Omega^{(2)}_{j,k} = \phi_{j,k}^{-1}(\Omega^{(2)})$. Switching back to coordinates $(x,y)$ and using \eqref{eq: density appears naturally}, we deduce the following estimates: $$\begin{aligned}\|Y^l \partial_t^m V_\Lambda^{j,k}\|_{L^2(\Omega^{(2)})}^2 &= \|\partial_t^m v_\Lambda\|_{L^2_g\left(\Omega^{(2)}_{j,k}\right)}^2, & (l, m) \in \{0, 1, 2\}^2, \\ 
\|Y^l \nabla_{X,Y} \partial_t^m V_\Lambda^{j,k}\|_{L^2(\Omega^{(2)})}^2 & \simeq \|y \nabla_{x,y} \partial_t^m v_\Lambda\|_{L^2_g\left(\Omega^{(2)}_{j,k}\right)}^2, & (l,m) \in \{0, 1, 2\}^2, \\
\|Y^l \Delta_{X,Y} V_\Lambda^{j,k}\|_{L^2(\Omega^{(2)})}^2 &\simeq \|y^2 \Delta_{x,y} v_\Lambda\|_{L^2_g\left(\Omega^{(2)}_{j,k}\right)}^2, & l \in \{1,2\}, \\
\|\nabla_{X,Y} (Y^2 \Delta_{X,Y} V_\Lambda^{j,k}) \|_{L^2(\Omega^{(2)})}^2 &\simeq \|y \nabla_{x,y} (y^2 \Delta_{x,y} v_\Lambda) \|_{L^2_g\left(\Omega^{(2)}_{j,k}\right)}^2.
\end{aligned}$$

Combining \eqref{eq: elliptic regularity}, \eqref{eq: time derivative term bound 1} and \eqref{eq: gradient term bound 1} and summing over $(j,k) \in \mathbb{Z}^2$ then yields \begin{multline}
    \label{eq: summed inequality} \sum_{j,k} \sup_\Omega|\nabla_{t,X,Y} V_\Lambda^{j,k}|^2 \lesssim \left(\sum_{m=1}^3 \|\partial_t^m v_\Lambda\|_{L^2_g\left((-T_2,T_2) \times \mathbb{H}^2\right)}^2\right) + \left(\sum_{m=0}^2 \|y \nabla_{x,y} \partial_t^m v_\Lambda\|_{L^2_g\left((-T_2,T_2) \times \mathbb{H}^2\right)}^2 \right) \\+ \|\Delta_g v_\Lambda\|_{L^2_g\left((-T_2,T_2) \times \mathbb{H}^2\right)}^2 + \|\Delta_g \partial_t v_\Lambda\|_{L^2_g\left((-T_2,T_2) \times \mathbb{H}^2\right)}^2 + \|y\nabla_{x,y} \Delta_g v_\Lambda \|_{L^2_g\left((-T_2,T_2) \times \mathbb{H}^2\right)}^2.
\end{multline}

Recall that \begin{equation}\label{eq: spawning powers of lambda}\partial_t^p \left( (-\Delta_g)^{q/2} v_\Lambda \right) = \int_0^\Lambda \lambda^{p+q-1} \sinh^{(p)}(\lambda t) dm_\lambda u, \quad p, q \in \mathbb{N}\end{equation} so by \eqref{eq: functional calculus bound} the first, third and fourth terms in \eqref{eq: summed inequality} can be bounded by $$Ce^{C\Lambda} \| \Pi_\Lambda u \|_{L^2_g\left(\mathbb{H}^2\right)}.$$ There remains to bound the terms involving a gradient, using integration by parts in the $(X,Y)$ variables:

\begin{lemma}\label{lemma: gradient}
Call $h$ any of the functions $v_\Lambda, \partial_t v_\Lambda, \partial_t^2 v_\Lambda, \Delta_g v_\Lambda$, then $$\|y\nabla_{x,y} h \|_{L^2_g((-T_3, T_3) \times \mathbb{H}^2)}^2 \lesssim \|h\|_{L^2_g((-T_3, T_3) \times \mathbb{H}^2)}^2 + \|\Delta_g h\|_{L^2_g((-T_3,T_3) \times \mathbb{H}^2)}^2.$$\end{lemma}

\begin{proof}[Proof of Lemma \ref{lemma: gradient}]
    Recall that $$\|y\nabla_{x,y} h \|_{L^2_g((-T_3, T_3) \times \mathbb{H}^2)}^2 \leq \sum_{j,k} \|\nabla_{X,Y} (h \circ \phi_{j,k}^{-1})\|_{L^2(\Omega^{(2)})}^2.$$ Let $\Tilde{\chi}(X,Y)$ be a cutoff function equal to 1 over $(-2-2\eta,2+2\eta)_X \times (\frac{1}{4}-2\eta, \frac{7}{4}+ 2\eta)_Y$, supported in $(-2-3\eta,2+3\eta)_X \times (\frac{1}{4}-3\eta, \frac{7}{4}+3\eta)_Y$. Then by integration by parts and the Cauchy--Schwarz inequality, denoting $\Omega^{(3)}=(-T_3,T_3) \times (-2-3\eta,2+3\eta)_X \times (\frac{1}{4}-3\eta, \frac{7}{4}+3\eta)_Y$ and $H_{j,k} = h \circ \phi_{j,k}^{-1}$, we get 
    $$\begin{aligned}
        \|\nabla_{X,Y} H_{j,k}\|_{L^2(\Omega^{(2)})}^2 &\leq \|\nabla_{X,Y} (\Tilde{\chi} H_{j,k})\|_{L^2(\Omega^{(3)})}^2 \\
        &\leq \|\Tilde{\chi} H_{j,k}\|_{L^2(\Omega^{(3)})} \|\Delta_{X,Y} (\Tilde{\chi} H_{j,k})\|_{L^2(\Omega^{(3)})} \\
        &\leq C \|H_{j,k}\|_{L^2(\Omega^{(3)})} \\
        &\times \left(\|H_{j,k}\|_{L^2(\Omega^{(3)})} + \|\nabla_{X,Y} H_{j,k}\|_{L^2(\Omega^{(3)})} + \|\Delta_{X,Y} H_{j,k}\|_{L^2(\Omega^{(3)})}\right)
    \end{aligned}$$

    Changing variables and summing over $(j,k)$ gives \begin{multline}\|y\nabla_{x,y} h \|_{L^2_g((-T_3, T_3) \times \mathbb{H}^2)}^2 \leq C \left( \|h\|_{L^2_g((-T_3, T_3) \times \mathbb{H}^2)}^2 \right.\\ \left.+ \|y\nabla_{x,y}h\|_{L^2_g((-T_3, T_3) \times \mathbb{H}^2)} \|h\|_{L^2_g((-T_3, T_3) \times \mathbb{H}^2)} + \|\Delta_g h\|_{L^2_g((-T_3, T_3) \times \mathbb{H}^2)}^2\right).\end{multline}
    The result then comes from the inequality $ab \leq \frac{1}{2} (\varepsilon a^2 + \frac{1}{\varepsilon} b^2)$ with $\varepsilon = \frac{1}{2C}$, applied to $a = \|y\nabla_{x,y}h\|_{L^2_g((-T_3, T_3) \times \mathbb{H}^2)}$ and $b = \|h\|_{L^2_g((-T_3, T_3) \times \mathbb{H}^2)}$.
\end{proof}

Combining this lemma with \eqref{eq: summed inequality}, we can replace all terms in the right-hand-side of \eqref{eq: summed inequality} by squared $L^2_g$-norms of terms of the form \eqref{eq: spawning powers of lambda} with $0 \leq p+q \leq 4$. We then conclude the proof of Lemma \ref{lemma: sobolev injection} using estimate \eqref{eq: functional calculus bound} and the inequality $$\sup_{0 \leq m \leq 4, 0\leq \lambda \leq \Lambda, |t| \leq T_3} \lambda^m \left(\cosh(\lambda t) + |\sinh(\lambda t)|\right) \leq Ce^{K\Lambda}$$ for some constants $C, K > 0$ independent of $\Lambda$ but allowed to depend on $T_3$. This completes the proof that Assumption \ref{ass: geodesic balls} implies the spectral estimate \eqref{eq: main result estimate}.

\subsection{Necessary condition}
\label{sec: necessary condition}

We now prove that Assumption \ref{ass: geodesic balls} is necessary for the spectral estimate to hold, using the heat kernel approach developed in \cite{wangwangzhangzhang2019spectralineq_heat} for the flat Laplacian. The first step is the following observability estimate:

\begin{lemma}[Spectral estimate implies observability]
    \label{lemma: observability}
    Consider $\omega \subset \mathbb{H}^2$ a measurable set such that the spectral estimate \eqref{eq: main result estimate} holds for $\omega$. For any solution $u$ of the heat equation on $\mathbb{H}^2$ \begin{equation}
        \label{eq: heat equation}
        \partial_t u - \Delta_g u = 0 \text{ in } (0, +\infty)_t \times \mathbb{H}^2_z, \quad u(0,.) = u_0 \in L^2_g(\mathbb{H}^2),
    \end{equation} for any $T > 0$, there exists a positive constant $C_{obs}(T, \omega)$ such that \begin{equation}
        \label{eq: observability}
        \int_{\mathbb{H}^2} |u(T,z)|^2 d\vol_g(z) \leq C_{obs} \int_0^T \int_\omega |u(s,z)|^2 d\vol_g(z)ds.
    \end{equation}
\end{lemma} We provide a proof and references concerning this classical result in Appendix \ref{sec: Appendix observability} for the sake of completeness.

We now show that the observability estimate \eqref{eq: observability} implies Assumption \ref{ass: geodesic balls}. Consider the heat kernel on the hyperbolic half-plane: \begin{equation}
    \label{eq: heat kernel expression}
    H(t, z, z') = \frac{\sqrt{2}}{(4\pi t)^\frac{3}{2}} e^{-\frac{t}{4}} \int_{d_g(z,z')}^\infty \frac{s e^{-\frac{s^2}{4t}}ds}{\left(\cosh(s) - \cosh(d_g(z,z'))\right)^{1/2}}, \quad \forall z, z' \in \mathbb{H}^2, t > 0.
\end{equation}

Any solution $u$ of the heat equation \eqref{eq: heat equation} then satisfies \begin{equation}
    \label{eq: general heat solution via heat kernel}
    u(t,z) = \int_{\mathbb{H}^2} H(t,z,z') u_0(z') d\vol_g(z'),
\end{equation} and by the semi-group property for $e^{t\Delta_g}$, one has \begin{equation} \label{eq: convolution heat kernels} H(t, z_1, z_2) = \int_{\mathbb{H}^2} H(s, z_1, z) H(t-s, z, z_2) d\vol_g(z), \quad \forall z_1, z_2 \in \mathbb{H}^2, 0 < s < t.\end{equation} We refer to \cite{mckean1970spectrumnegativecurvature,grigor'yannoguchi1998heatkerhyperbolic} for a proof of the expression of the heat kernel on the hyperbolic plane and to \cite[Chapters 4, 7]{chengliyau2019estimate_heatkernel_mfd,grigor'yan2009book} and the references therein for a more detailed study of the heat kernel on manifolds.

We now fix some $z_0 \in \mathbb{H}^2$ and show that there exist some $R, \delta > 0$ independent of $z_0$ such that $\vol_g\left( B^g_R(z_0) \cap \omega \right) \geq \delta$. Let us apply Lemma \ref{lemma: observability} with the initial condition $$u_0(z) = H(1, z, z_0).$$ By \eqref{eq: convolution heat kernels}, the corresponding solution $u(t, z)$ of the heat equation is \begin{equation}
    \label{eq: heat solution for observability}
    u(t,z) = H(t+1, z, z_0), \quad t > 0, z \in \mathbb{H}^2.
\end{equation} Inserting \eqref{eq: heat solution for observability} in the observability inequality \eqref{eq: observability} for time $T = 1$ yields \begin{equation}
    \label{eq: observability for specific solution}
    \int_{\mathbb{H}^2} |H(2, z, z_0)|^2 d\vol_g(z) \leq C_{obs}(1) \int_0^1 \int_\omega |H(s+1, z, z_0)|^2 d\vol_g(z) ds.
\end{equation}

The next lemma allows to conclude: \begin{lemma}[Gaussian bounds]
    \label{lemma: gaussian bounds heat kernel}
    The heat kernel $H$ given by \eqref{eq: heat kernel expression} admits Gaussian (as functions of $d_g(z,z')$) lower and upper bounds. More precisely, there exist some constants $\alpha, \beta > 0$ such that the lower bound \begin{equation}
        \label{eq: gaussian lower bound}
        H(2, z, z_0) \gtrsim e^{-\beta d_g(z,z_0)^2}
    \end{equation} and the upper bound \begin{equation}
        \label{eq: gaussian upper bound} H(t, z, z_0) \lesssim \frac{\sqrt{\gamma t}}{f(\gamma t)} e^{-\frac{\alpha}{t} d_g(z, z_0)^2}, \quad t > 0,
    \end{equation} hold for some $\gamma > 0$ and $f(t) = e^\frac{t}{4} t^\frac{3}{2}$.
\end{lemma} We deduce the necessary condition of Theorem \ref{theo: main result} from the lemma, which will be proved afterwards. The l.h.s. of \eqref{eq: observability for specific solution} can be bounded from below as follows, with $C$ a positive constant $$\begin{aligned}\int_{\mathbb{H}^2} |H(2, z, z_0)|^2 d\vol_g(z) &\geq C \int_{\mathbb{H}^2} e^{-2\beta d_g(z,z_0)^2} d\vol_g(z).\end{aligned}$$ The integral in the right-hand-side is independent of $z_0$, as shown by the change of variables $$z' = y_0^{-1}(z - (x_0,0)).$$ Indeed, the hyperbolic distance and the measure $d\vol_g$ are invariant by scaling and horizontal translation (see for example \cite[(2.6) and (2.8)]{borthwick2016book}), so that $d_g(z, z_0) = d_g(z', (0,1))$ and $d\vol_g(z) = d\vol_g(z')$.

We now bound the r.h.s. of \eqref{eq: observability for specific solution} from above using the upper bound \eqref{eq: gaussian upper bound}. We have $$\begin{aligned}
    \int_0^1 \int_\omega |H(s+1, z, z_0)|^2 d\vol_g(z) ds &\leq C' \int_0^1 \int_\omega \frac{(s+1)\gamma}{f((s+1)\gamma)^2} e^{-2\frac{\alpha}{s+1} d_g(z,z_0)^2} d\vol_g(z) ds \\
    &\leq\frac{C'\gamma}{f(\gamma)^2} \int_\omega e^{-\alpha d_g(z,z_0)^2} d\vol_g(z).
\end{aligned}$$ We introduce a parameter $L > 0$ and split the domain of integration into $B^g_L(z_0)$ and $B^g_L(z_0)^c$: $$\begin{aligned}
    &\int_0^1 \int_\omega |H(s+1, z, z_0)|^2 d\vol_g(z) \\
    &\quad \leq \frac{C'\gamma}{f(\gamma)^2} \left( \int_{\omega \cap B^g_L(z_0)} e^{-\alpha d_g(z,z_0)^2} d\vol_g(z) + \int_{\omega \cap {B^g_L(z_0)}^c} e^{-\alpha d_g(z,z_0)^2} d\vol_g(z) \right) \\
    &\quad \leq \frac{C' \gamma}{f(\gamma)^2} \left(\vol_g(\omega \cap B^g_L(z_0)) + e^{-\frac{\alpha}{2} L^2} \int_{\mathbb{H}^2} e^{-\frac{\alpha}{2} d_g(z,z_0)^2} d\vol_g(z)\right).
\end{aligned}$$

The first term is precisely the volume we wish to bound from below to prove Assumption~\ref{ass: geodesic balls}. The integral in the second term is independent from $z_0$ for the same reason as before.

Plugging the lower and upper bounds together gives \begin{equation}
    \label{eq: lower < upper}
    C'' \leq \vol_g(\omega \cap B^g_L(z_0)) + e^{-\frac{\alpha}{2}L^2} \int_{\mathbb{H}^2} e^{-\frac{\alpha}{2} d_g(z',(0,1))^2} d\vol_g(z')
\end{equation} where $C'' = \frac{Cf(\gamma)^2}{C'\gamma} \int_{\mathbb{H}^2} e^{-2\beta d_g(z',(0,1))^2} d\vol_g(z')$ does not depend on $z_0$. We fix $L$ large enough for the second term of the r.h.s. to be smaller than $\frac{C''}{2}$. Assumption \ref{ass: geodesic balls} then holds for $R = L$ and $\delta = \frac{C''}{2}$.

\begin{proof}[Proof of lemma \ref{lemma: gaussian bounds heat kernel}]
    There remains to prove Gaussian lower and upper bounds for the heat kernel $H(t,z,z')$. We start with the lower bound. The explicit expression \eqref{eq: heat kernel expression} gives $$\begin{aligned}
        H(2,z,z') &\gtrsim \int_{d_g(z,z')}^\infty \frac{se^{-\frac{s^2}{8}}ds}{\left(\cosh(s) - \cosh(d_g(z,z'))\right)^\frac{1}{2}} ds \\
        &\gtrsim \int_{d_g(z,z')}^\infty \frac{se^{-\frac{s^2}{8}}ds}{\left(\cosh(s) - 1\right)^\frac{1}{2}} ds \\
        &\gtrsim \int_{d_g(z,z')}^\infty \frac{se^{-\frac{s^2}{8}}ds}{\sinh(s/2)}.
    \end{aligned}$$ Then since $\sinh(\frac{s}{2}) \leq \frac{1}{2} e^{\frac{s}{2}}$ and $\frac{s}{2} + \frac{s^2}{8} \leq \frac{1+s^2}{2}$, we get the lower bound $$H(2,z,z') \gtrsim \int_{d_g(z,z')}^\infty se^{-\frac{s^2}{2}} ds \gtrsim e^{-\frac{1}{2} d_g(z,z')^2}.$$

    To prove the upper bound, we employ the strategy developed by Grigor'yan in \cite[Section 5.1]{grigor'yan1997gaussianbounds,grigor'yan1999estimates_heatkernel_mfds}, which consists in deducing off-diagonal estimates from diagonal estimates. $H$ satisfies the following diagonal estimate: \begin{equation}
        \label{eq: diagonal estimate}
        \begin{aligned}
        H(t,z,z) &= \frac{\sqrt{2}}{(4\pi)^\frac{3}{2}f(t)} \int_0^\infty \frac{s e^{-\frac{s^2}{4t}}}{(\cosh(s) - 1)^\frac{1}{2}}ds \\
        &\lesssim \frac{1}{f(t)} \int_0^\infty \frac{s e^{-\frac{s^2}{4t}}}{\sinh(\frac{s}{2})} ds \\
        &\lesssim \frac{1}{f(t)} \int_0^\infty e^{-\frac{s^2}{4t}} ds \text{ as } \sinh\left(\frac{s}{2}\right) \geq \frac{s}{2} \text{ for positive } s \\
        &\lesssim \frac{t^\frac{1}{2}}{f(t)}.
        \end{aligned}
    \end{equation} where $f(t)$ is the function introduced in Lemma \ref{lemma: gaussian bounds heat kernel}. The estimate is independent of $z$.

    The upper bound can then be derived from \cite[Theorem 1.1]{grigor'yan1997gaussianbounds}, the assumption of which is satisfied by $\frac{f(t)}{t^\frac{1}{2}} = t e^{\frac{t}{4}}$ (see Section 2, Example 2.). Thus, there exist some positive constants $\alpha, K, \gamma$ such that for any $(z, z') \in \left(\mathbb{H}^2\right)^2$, one has $$H(t,z,z') \leq \frac{K(\gamma t)^\frac{1}{2}}{f(\gamma t)} e^{-\frac{\alpha}{t}d_g(z,z')^2}.$$ This completes the proof of the lemma.
\end{proof}

\appendix

\section{Logunov and Malinnikova's propagation of smallness inequality}
\label{sec: Appendix Logunov}

We recall the precise statement of the estimate by Logunov and Malinnikova which we use in the proof of Theorem \ref{theo: main result}, namely \cite[Theorem 5.1]{logunovmalinnikova2018qtttve_propag_smallness}: 

\begin{theorem}[\cite{logunovmalinnikova2018qtttve_propag_smallness}]
\label{theo: logunov-malinnikova}
    Consider $\Omega$ a bounded domain in $\mathbb{R}^d$. There exists a constant $c \in (0,1)$ that depends only on the dimension $d$ such that the following holds:
    take any $u$ solution of an elliptic equation in divergence form \begin{equation}
        \Div(A\nabla u) = 0 \text{ over } \Omega
    \end{equation} where $A = (a_{i,j}(x))_{i,j}$ is a symmetric, uniformly elliptic matrix with Lipschitz entries: $$\Lambda_1^{-1} |\zeta|^2 \leq \langle A \zeta, \zeta \rangle \leq \Lambda_1 |\zeta|^2, \quad |a_{i,j}(x) - a_{i,j}(y) | \leq \Lambda_2 |x-y|.$$ Let $m, \delta, \rho$ be positive numbers, and assume sets $E, \mathcal{K} \subset \Omega$ satisfy $$C_\mathcal{H}^{n-1-c+\delta}(E) > m, \quad \dist(E, \partial\Omega) > \rho, \dist(\mathcal{K}, \partial\Omega) > \rho.$$ Then there exist $C, \alpha > 0$ depending only on $m, \delta, \rho, \Lambda_1, \Lambda_2, \Omega$ such that $$\sup_{\mathcal{K}} |\nabla u| \leq C \sup_{E}|\nabla u|^\alpha \sup_{\Omega}|\nabla u|^{1-\alpha}.$$ 
\end{theorem} Taking $\delta = c$, the result holds if $E$ is a set of Hausdorff dimension $n-1$.

\section{Spectral estimate implies observability}
\label{sec: Appendix observability}

In this Appendix, we prove that on any set where the spectral estimate \eqref{eq: main result estimate} holds, solutions of the heat equation \eqref{eq: heat equation} satisfy the observability inequality \eqref{eq: observability}. The link between both was established by Miller \cite{miller2010directlebeaurobbiano} while Phung and Wang established observability from any set of positive measure \cite{phungwang2013observability}. The present proof follows closely that of \cite[Lemmas 2.3, 2.4]{wangwangzhangzhang2019spectralineq_heat} and we present it for the sake of completeness.

The proof is divided into two parts. In the first one, we use the spectral inequality to prove the interpolation estimate \eqref{eq: Hölder}. Differences with \cite{wangwangzhangzhang2019spectralineq_heat} are that the Plancherel-based estimates for Fourier multipliers are replaced by inequality \eqref{eq: functional calculus bound}, and that the authors prove the inequality for any $\theta \in (0,1)$ while we focus on the case $\theta = 2$. In the second part, observability is deduced from the interpolation inequality using a telescoping series method: there we follow \cite{wangwangzhangzhang2019spectralineq_heat} in the simpler case where $F$ equals $(0,T)$ rather than any positive-measure subset.

\begin{proof}[Proof of Lemma \ref{lemma: observability}]
    Assume that \eqref{eq: main result estimate} holds for the set $\omega$ and consider a solution $u$ of the heat equation \eqref{eq: heat equation} on $\mathbb{H}^2$. Fix $T > 0$ and denote $u_0 = u(0,\cdot)$. Then, $$u(T, z) = e^{T\Delta_g}u_0(z), \quad z \in \mathbb{H}^2.$$ We first prove that $u$ satisfies the following Hölder-type inequality for some constant $\Tilde{C}$: \begin{equation}
        \label{eq: Hölder}
        \int_{\mathbb{H}^2} |u(T,z)|^2 d\vol_g \leq \exp\left(\Tilde{C}\left(1+\frac{1}{T}\right)\right) \left( \int_\omega |u(T,z)|^2 d\vol_g\right)^\frac{1}{2} \left(\int_{\mathbb{H}^2} |u_0(z)|^2 d\vol_g \right)^\frac{1}{2}.
    \end{equation}
    
    Denote $\Pi_{>\Lambda}$ the spectral projector defined by $$\Pi_{>\Lambda} u = \int_\Lambda^{+\infty} dm_\lambda u, \quad u \in L^2_g(\mathbb{H}^2),$$ then $u_0 = \Pi_\Lambda u_0 + \Pi_{>\Lambda} u_0$. By the spectral estimate \eqref{eq: main result estimate}, there exists some $K > 0$ such that $$\begin{aligned}
        \int_{\mathbb{H}^2} |u(T,z)|^2 d\vol_g &\leq 2 \int_{\mathbb{H}^2} |e^{T\Delta_g} \Pi_\Lambda u_0(z)|^2 d\vol_g + 2 \int_{\mathbb{H}^2} |e^{T\Delta_g} \Pi_{>\Lambda} u_0(z)|^2 d\vol_g \\
        &\leq 2Ke^{K\Lambda} \int_\omega |e^{T\Delta_g} \Pi_\Lambda u_0(z)|^2 d\vol_g + 2 \int_{\mathbb{H}^2} |e^{T\Delta_g} \Pi_{>\Lambda} u_0(z)|^2 d\vol_g \\
        &\leq 4Ke^{K\Lambda} \int_\omega |e^{T\Delta_g} u_0(z)|^2 d\vol_g + (2+4Ke^{K\Lambda}) \int_{\mathbb{H}^2} |e^{T\Delta_g} \Pi_{>\Lambda} u_0(z)|^2 d\vol_g.
    \end{aligned}$$ Since $\int_{\mathbb{H}^2}|e^{T\Delta_g} \Pi_{>\Lambda} u_0(z)|^2 d\vol_g \leq e^{-T\Lambda^2} \|\Pi_{>\Lambda} u_0\|_{L^2_g(\mathbb{H}^2)}^2 \leq e^{-T\Lambda^2} \|u_0\|_{L^2_g(\mathbb{H}^2)}^2$, it follows that \begin{equation}
        \begin{aligned}\int_{\mathbb{H}^2} |u(T,z)|^2 d\vol_g &\leq 4Ke^{K\Lambda} \int_\omega |e^{T\Delta_g} u_0(z)|^2 d\vol_g + (2+4Ke^{K\Lambda}) e^{-T\Lambda^2} \|u_0\|_{L^2_g(\mathbb{H}^2)}^2 \\
        &\leq 6K \left(e^{K\Lambda} \int_\omega |e^{T\Delta_g} u_0(z)|^2 d\vol_g + e^{K\Lambda - T\Lambda^2} \|u_0\|_{L^2_g(\mathbb{H}^2)}^2\right)\end{aligned}
    \end{equation} Fix $\eta \in (0,1)$ and $\Lambda$ such that $e^{K\Lambda - T\Lambda^2} = \eta$, then one has $$\Lambda = \frac{K + \sqrt{K^2 + 4T\log\frac{1}{\eta}}}{2T} \leq \frac{1}{T} \left(K + \sqrt{T \log\frac{1}{\eta}}\right).$$ Thus, $$\begin{aligned}
        e^{K\Lambda} &\leq \exp\left(\frac{K^2}{T}\right) \exp\left(\frac{K}{\sqrt{T}}\sqrt{\log\frac{1}{\eta}}\right) \\
        &\leq \exp\left(\frac{K^2}{T}\right) \exp\left(\frac{K^2}{2T} +\frac{1}{2}\log\frac{1}{\eta}\right) \\
        &\leq \exp\left(\frac{3K^2}{2T}\right) \eta^{-1}.
    \end{aligned}$$
    Setting $$\eta = \left(\frac{\int_\omega |u(T,z)|^2 d\vol_g}{\int_{\mathbb{H}^2} |u_0(z)|^2 d\vol_g}\right)^\frac{1}{2} < 1$$ yields estimate \eqref{eq: Hölder}.

    Then, by \eqref{eq: Hölder} and translation in time, one gets that for $0 < t_1 < t_2$ and $\varepsilon > 0$, \begin{equation} \label{eq: translated in time}
    \int_{\mathbb{H}^2} |u(t_2,z)|^2 d\vol_g \leq \frac{1}{\varepsilon}e^{2\Tilde{C}\left(1+\frac{1}{t_2-t_1}\right)} \int_\omega |u(t_2,z)|^2 d\vol_g + \varepsilon \int_{\mathbb{H}^2} |u(t_1,z)|^2 d\vol_g.\end{equation} Now set $l_1 = T, \lambda \in (\frac{1}{\sqrt{2}},1)$ and $l_{m+1} = \lambda^m l_1$ for $m \geq 0$. Take an integer $m$ and some $s$ such that $$0 < l_{m+2} < l_{m+1} \leq s < l_m \leq T,$$ then by \eqref{eq: translated in time}, $$\int_{\mathbb{H}^2} |u(l_m,z)|^2 d\vol_g \leq \frac{1}{\varepsilon} e^{2\Tilde{C}\left(1 + \frac{1}{l_{m+1} - l_{m+2}}\right)} \int_\omega|u(s,z)|^2 d\vol_g + \varepsilon \int_{\mathbb{H}^2} |u(l_{m+2},z)|^2 d\vol_g.$$ 
    Integrating over $s \in (l_{m+1},l_m),$ we get \begin{multline}\int_{\mathbb{H}^2} |u(l_m,z)|^2 d\vol_g \\ \leq \varepsilon \int_{\mathbb{H}^2} |u(l_{m+2},z)|^2 d\vol_g + \frac{1}{\varepsilon} \frac{1}{l_m-l_{m+1}} e^{2\Tilde{C}\left(1 + \frac{1}{l_{m+1} - l_{m+2}}\right)} \int_{l_{m+1}}^{l_m}\int_\omega|u(s,z)|^2 d\vol_g ds.\end{multline} Using the inequality $l_m - l_{m+1} \geq e^{-\frac{1}{l_m-l_{m+1}}}$, and the identities $$l_m-l_{m+1} = \frac{1}{1+\lambda}(l_m - l_{m+2}), \quad l_{m+1} - l_{m+2} = \frac{\lambda}{1+\lambda} (l_m-l_{m+2}),$$ we obtain  \begin{multline}
    \label{eq: obs l_m à télescoper}
    \int_{\mathbb{H}^2} |u(l_m,z)|^2 d\vol_g \\ \leq \varepsilon \int_{\mathbb{H}^2} |u(l_{m+2},z)|^2 d\vol_g + \frac{1}{\varepsilon} e^{2\Tilde{C}} e^{\frac{C'}{l_m - l_{m+2}}} \int_{l_{m+1}}^{l_m}\int_\omega|u(s,z)|^2 d\vol_g ds\end{multline} with $C' = 1 + \lambda + \frac{2\Tilde{C}(1+\lambda)}{\lambda}$. \eqref{eq: obs l_m à télescoper} can be rewritten as \begin{multline}
        \varepsilon e^{-\frac{C'}{l_m - l_{m+2}}} \int_{\mathbb{H}^2} |u(l_m,z)|^2 d\vol_g - \varepsilon^2 e^{-\frac{C'}{l_m - l_{m+2}}} \int_{\mathbb{H}^2} |u(l_{m+2},z)|^2 d\vol_g \\ \leq e^{2\Tilde{C}} \int_{l_{m+1}}^{l_m}\int_\omega|u(s,z)|^2 d\vol_g ds.
    \end{multline} Setting $\mu:= \frac{1}{2-\lambda^{-2}} > 1$ and $\varepsilon = \exp\left[-\frac{(\mu-1)C'}{l_m-l_{m+2}}\right]$ yields \begin{multline}
        e^{-\frac{\mu C'}{l_m - l_{m+2}}} \int_{\mathbb{H}^2} |u(l_m,z)|^2 d\vol_g - e^{-\frac{(2\mu-1)C'}{l_m - l_{m+2}}} \int_{\mathbb{H}^2} |u(l_{m+2},z)|^2 d\vol_g \\ \leq e^{2\Tilde{C}} \int_{l_{m+1}}^{l_m}\int_\omega|u(s,z)|^2 d\vol_g ds.
    \end{multline} Now $$\frac{2\mu-1}{l_m-l_{m+2}} = \frac{\mu}{\lambda^2 (l_m - l_{m+2})} = \frac{\mu}{l_{m+2}-l_{m+4}},$$ hence \begin{multline}
        e^{-\frac{\mu C'}{l_m-l_{m+2}}} \int_{\mathbb{H}^2} |u(l_m,z)|^2 d\vol_g - e^{-\frac{\mu C'}{l_{m+2}-l_{m+4}}} \int_{\mathbb{H}^2} |u(l_{m+2},z)|^2 d\vol_g \\ \leq e^{2\Tilde{C}} \int_{l_{m+1}}^{l_m} \int_\omega |u(s,z)|^2 d\vol_g ds.
    \end{multline} Summing this inequality over all odd integers $m$, we obtain $$\
        e^{\frac{\mu C'}{l_1-l_3}} \int_{\mathbb{H}^2} |u(l_1,x)|^2 dx \leq e^{2\Tilde{C}} \int_0^1 \int_\omega |u(s,z)|^2 d\vol_g ds,$$ which in turn gives the observability estimate \eqref{eq: observability}.
\end{proof}

\section*{Acknowledgements}

The author would like to thank Nicolas Burq and Alix Deleporte for insightful discussions about the topic of this article. This work was funded by a CDSN PhD grant from Ecole Normale Supérieure Paris-Saclay via the Hadamard Doctoral School of Mathematics (EDMH).

\bibliographystyle{alpha}
\bibliography{bibli}

\end{document}